\documentclass{amsart}
\usepackage{amsfonts}
\usepackage{amsmath,amssymb}	
\usepackage{amsthm}
\usepackage{amscd}
\usepackage[dvipdfmx]{color}
\usepackage{graphics}
\usepackage[dvipdfmx]{graphicx}
\theoremstyle{definition}{
\newtheorem{Def}{{\rm Definition}}
\newtheorem{Ex}{{\rm Example}}
\newtheorem{Rem}{{\rm Remark}}
\newtheorem*{Prob}{{\rm Problem}}
\newtheorem*{MainProb}{Main Problem}
}
\theoremstyle{plain}
{

\newtheorem{Thm}{Theorem}

}

\begin{document}
\title[Good differentiable maps locally like natural maps to Reeb graphs]{Maps on manifolds onto graphs locally regarded as the quotient maps onto Reeb spaces of some differentiable maps and a new construction problem}

\author{Naoki Kitazawa}
\address{Institute of Mathematics for Industry, Kyushu University, 744 Motooka, Nishi-ku Fukuoka 819-0395, Japan}
\email{n-kitazawa@imi.kyushu-u.ac.jp}
\maketitle
\begin{abstract}
The {\it Reeb space} of a function or a map on a manifold is defined as the space of all connected components of preimages and represents the manifold compactly. In fact, Reeb spaces are fundamental and useful tools in geometric theory of so-called {\it Morse} functions and more general maps which are sufficiently tame. 

Can we construct an explicit good function inducing a given graph as the Reeb space ({\it Reeb graph})? These problems were launched by Sharko in 2000s and have been explicitly solved by several researchers. 
As related pioneering studies, the author also found and solved problems adding constraints on singularities and preimages for example. 

The present paper concerns new problems on these works. We define the classes of maps onto graphs locally regarded as ones onto the Reeb spaces induced from smooth functions of suitable classes and consider and challenge the problems for the classes. 
\end{abstract}

\section{Introduction.}
\label{sec:1}
\subsection{Reeb spaces and graphs and differentiable functions realizing given graphs as Reeb graphs}
The {\it Reeb space} of a continuous map of a suitable class on a topological space is the space of all connected components of preimages. For a differentiable function, consider the set of all points in the Reeb space coinciding with the set of all connected components of preimages containing some {\it singular} points: a {\it singular point} of a differentiable map is a point at which the rank of the differential is smaller than both the dimensions of the manifolds of the domain and the target. For {\it so-called} Morse functions, functions with finitely many singular points on closed manifolds, and more general functions of several suitable classes, the spaces are graphs such that the sets of all the vertices (the {\it vertex set}) are the sets defined before. They are the {\it Reeb graphs} of the maps. See \cite{reeb} and \cite{sharko} for example.

Reeb graphs and spaces are fundamental and important in the algebraic topological or differential topological theory of Morse functions and their variants.

We introduce several terminologies and a problem on construction of good (differentiable) functions inducing Reeb graphs isomorphic to given graphs.

The {\it singular set} of a differentiable map is the set of all singular points of the map. A {\it singular value} is a point in the manifold of the target such that the preimage contains some singular points. A {\it regular value} is a point in the manifold of the target of the map which is not a singular value. The {\it singular value set} is the image of the singular set of the map, or equivalently, the set of all singular values of it.

\begin{Prob}
\label{prob}
Can we construct a differentiable function with good geometric properties inducing a given graph as the Reeb graph? 
\end{Prob}
Note that for example, we do not fix a manifold on which we construct a desired function.

A problem of this type was first considered and explicitly solved by Sharko (\cite{sharko}). \cite{batistacostamezasarmiento}, \cite{martinezalfaromezasarmientooliveira}, \cite{masumotosaeki} and \cite{michalak} are important studies related to this. 

Later the author set and solved explicit cases in \cite{kitazawa4} and \cite{kitazawa5} and \cite{saeki4} is also regarded as a paper motivated by them. Different from the other existing studies, conditions on preimages of regular values are posed and manifolds appearing there may not be spheres, for example. 

\subsection{Pseudo quotient maps.}
A {\it pseudo quotient} map on a differentiable manifold is a surjective continuous map onto a lower dimensional polyhedron, and defined as a map locally regarded as the natural quotient map onto the Reeb space of a differentiable map of a suitable class. They were first defined by Kobayashi and Saeki in 1996 (\cite{kobayashisaeki}) as useful objects in the theory of so-called {\it stable} maps and {\it generic} smooth maps from manifolds whose dimensions are greater than $2$ into the plane. Later the author used these objects in new explicit situations starting from redefining them in \cite{kitazawa2} and \cite{kitazawa3} for example.

\subsection{The content of the present paper.}
In the present paper, we discuss the following. 
\begin{itemize}
\item First we consider suitable classes of continuous or differentiable maps on differentiable manifolds (with subsets of the manifolds of the domains). We introduce the {\it Reeb graphs} of these maps. This gives a refinement of the definition of a Reeb graph which has not ever appeared. Remark \ref{rem:2} presents the reason we introduce such notions.
\item Second we redefine {\it pseudo quotient} maps on differentiable manifolds of a class of maps just before. 
\item We propose a variant of our explicit construction problem (Problem) before as Main Problem and give an answer as a main theorem (Theorem \ref{thm:1}) with several terminologies, notions, and notation needed. This is a problem of a new type and our answer is also a result of a new type. Theorem \ref{thm:2} also presents a related answer and another main theorem.
\end{itemize}
\begin{MainProb}
\label{mainprob}
Can we construct a pseudo quotient map having good geometric properties onto a given graph? Moreover, is this map essentially the quotient map onto the Reeb graph induced from a smooth function of a natural class.
\end{MainProb}


 

\section{Classes of continuous or differentiable maps on differentiable manifolds and the Reeb graphs of the maps of these classes.}
\begin{Def}
\label{def:1}
For a $C^r$ manifold $X$, a subset $A \subset X$ is a {\it measure zero set} if for a family $\{(U_{\lambda} \subset X,{\phi}_{\lambda}:U_{\lambda} \rightarrow {\mathbb{R}}^{\dim X})\}_{\lambda \in \Lambda}$ of local coordinates satisfying $A= {\bigcup}_{\lambda \in \Lambda} U_{\lambda}$ and compatible with the $C^{r}$ differentiable structure, ${\phi}_{\lambda}(A \bigcap U_{\lambda}) \subset {\mathbb{R}}^{\dim X}$ is a Lebesgue measurable set and the Lebesgue measure is $0$ for every $\lambda \in \Lambda$. 
\end{Def}
\begin{Def}
\label{def:2}
A {\it graph} $K$ is an object represented as a pair of the set $V$ (the {\it vertex set}) and the set $E$ (the {\it edge set}) consisting of pairs of subsets of $V$ whose sizes are $1$ or $2$ and elements of a non-empty set $S_K$.
An edge $e=(V_e \subset V,s_k \in S_K)$ is a {\it loop} if the size of the subset $V_e \subset V$ is $1$. 
The graph $K=(V,E)$ is {\it finite} if both the vertex set and the edge set are finite sets.
A {\it subgraph} of a graph is a graph whose vertex set and edge set are subsets of the original vertex set and the original edge set.
\end{Def}
The graph is regarded as a $1$-dimensional cell complex.
We correspond a one-point set to each vertex and a closed interval to each edge and attach them in a natural way.
This is an elementary argument. The graph is {\it connected} if it is connected as a topological space.

In the present paper, graphs are finite and connected graphs with at least one edge unless otherwise stated. In this case, it is a $1$-dimensional, connected and compact polyhedron.

\begin{Def}
\label{def:3}
An {\it isomorphism} between two finite graphs is a (PL) homeomorphism between the graphs mapping the vertex set of a graph onto the vertex set of the other graph.
\end{Def}

We introduce the definition of the {\it Reeb space} of a map $c:X \rightarrow Y$ between two topological spaces. Let ${\sim}_c$ be a relation on $X$ defined by the following rule: $x_1 {\sim}_c x_2$ holds if and only if $x_1$ and $x_2$ are in a same connected component of some preimage $c^{-1}(y)$ for $y \in Y$. This is an equivalence relation on $X$.    
\begin{Def}
\label{def:4}
The quotient space $W_c:=X/{\sim}_c$ is the {\it Reeb space} of $c$.
\end{Def}

$q_c:X \rightarrow W_c$ denotes the quotient map onto $W_c$ and we can also define a map denoted by $\bar{c}:W_c \rightarrow Y$ and satisfying the relation $c=\bar{c} \circ q_c$ uniquely.
\begin{Def}
\label{def:5}
Let $(r,s)$ be a pair of non-negative integers satisfying $r>s$ or a pair such that $r=\infty$ and that $s$ is a non-negative integer.

Let $X$ be an $m$-dimensional $C^r$ manifold where $m$ is an integer greater than $1$ and $Y$ be a $1$-dimensional $C^r$ manifold. Assume that there exists a measure zero set $A \subset X$.
The pair $(f,A)$ of a map $f:X \rightarrow Y$ between the $C^r$ manifolds and $A$ is said to be a {\it $(C^r,C^s)$} map if $f$ is of class $C^r$ at any point in $X-A$ and of class $C^s$ at any point in $A$.
$A$ is called the {\it measure zero set} of $(f,A)$. 

Consider the Reeb space $W_f$ of the map $f$. Let $V$ be the set of all points $v \in W_f$ whose preimages ${q_f}^{-1}(v)$ contain some singular points of $f$ or points in $A$. If we can regard $W_f$ as a graph whose vertex set is $V$, then we call the graph the {\it Reeb graph} of $f$.

For the pair $(f,A)$, we omit $A$ and we use $f$ instead if we can guess $A$ easily.
\end{Def}

\section{A pseudo quotient map of a class of maps.}



\begin{Def}
\label{def:6}
Let $r$ be a positive integer or $\infty$. Let $X_1$ and $X_2$ be $C^r$ differentiable manifolds of dimension $m>1$ and $K_1$ and $K_2$ be graphs. 
Two continuous maps $c_1:X_1 \rightarrow K_1$ and $c_2:X_2 \rightarrow K_2$ such that the images are subgraphs of the given graphs are said to be {\it $C^r$-PL equivalent} or $c_1:X_1 \rightarrow K_1$ is {\it $C^r$-PL equivalent} to $c_2:X_2 \rightarrow K_2$ if there exist a $C^r$ diffeomorphism ${\phi}_X$ and an isomorphism ${\phi}_K$ satisfying the relation ${\phi}_K \circ c_1=c_2 \circ {\phi}_X$.
\end{Def}

\begin{Def}
\label{def:7}
Let $m>2$ be a positive integer.
Let $\mathcal{C}$ be a class of $(C^r,C^s)$ maps from $m$-dimensional manifolds into $1$-dimensional ones whose Reeb spaces are regarded as Reeb graphs. A continuous map $q$ on an $m$-dimensional $C^r$ manifold onto a graph $K$ is said to be a {\it pseudo quotient} map of the class $\mathcal{C}$ if the following properties hold.
\begin{itemize}
\item

 At each point $p$ in the interior of an edge $e$, consider a small closed interval $C_p$ containing the point in the interior and regarded as a graph with exactly one edge and two vertices canonically. $q {\mid}_{q^{-1}(C_p)}:q^{-1}(C_p) \rightarrow C_p$ is $C^r$-PL equivalent to a $C^r$ trivial bundle whose base space is a closed interval in the interior of an edge in the Reeb graph of a map of the class $\mathcal{C}$. We regard the base space of the latter trivial bundle as a natural graph with exactly one edge and two vertices. 
\item At each vertex $p$ of $e$, consider a small regular neighborhood $C_p$ containing the point and regarded as a graph with exactly $n(p)$ edges and $n(p)+1$ vertices and having $p$ as a vertex where $n(p)$ is the number of edges containing $p$. $q {\mid}_{q^{-1}(C_p)}:q^{-1}(C_p) \rightarrow C_p$ is $C^r$-PL equivalent to the PL map $q_{f_p} {\mid}_{{q_{f_p}}^{-1}({C^{\prime}}_p)}$ onto a suitable small regular neighborhood ${C^{\prime}}_p$ of a vertex $p^{\prime}$ regarded as a graph having $p^{\prime}$ as a vertex in the Reeb graph of a map $f_p$ of the class $\mathcal{C}$. Here the isomorphism ${\phi}_K$ between the graphs in the situation of Definition \ref{def:6} can be chosen as a map mapping $p$ to $p^{\prime}$.
\end{itemize}  
\end{Def}

In Definition \ref{def:7} vertices of $C_p$, ${C^{\prime}}_p$, and the base space of the $C^r$ trivial bundle in the first condition except the vertex $p$ and the vertex $p^{\prime}$ of the original Reeb graphs are originally in the interiors of edges of the Reeb graphs.


\begin{Def}
\label{def:8}
A pseudo quotient map of a class $\mathcal{C}$ of $(C^r,C^s)$ maps from $m$-dimensional $C^r$ manifolds onto graphs is said to be {\it realized} as a quotient map of the class $\mathcal{C}$ if there exists a map of the original class $\mathcal{C}$ such that the induced quotient map onto the Reeb graph of the map and the original pseudo quotient map are $C^r$-PL equivalent.
\end{Def}

\section{Main theorems and proofs with explanations of terminologies, notions and notation needed.}
We introduce main theorems, which are explicit answers to Main Problem.

We review {\it fold} maps and {\it special generic} maps. 

Let $r$ and $s$ satisfy either of the following three.
\begin{itemize}
\item $r \geq s \geq 0$ and these two numbers are integers.
\item $r=\infty$ and $s \geq  0$ be an integer.
\item $r=s=\infty$.
\end{itemize} 
For two maps $c_1:X_1 \rightarrow Y_1$ and $c_2:X_2 \rightarrow Y_2$ between $C^{r}$ manifolds, they are said to be {\it $C^s$ equivalent} or $c_1$ is {\it $C^{s}$ equivalent} to $c_2$ if there exists a pair $({\phi}_X,{\phi}_Y)$ of $C^s$ diffeomorphisms satisfying ${\phi}_Y \circ c_1=c_2 \circ {\phi}_X$. $c_1$ is {\it $C^s$ equivalent} to $c_2$ {\it around a point} $p_1 \in X_1$ if the following conditions hold.
\begin{itemize}
\item There exists a pair $(U_1,V_1)$ of open subsets of $X_1$ and $Y_1$ respectively.
\item $U_1 \ni p_1$ and $V_1 \supset c_1(U_1) \ni c_1(p_1)$.
\item There exist a pair $(U_2,V_2)$ of open subsets of $X_2$ and $Y_2$ respectively.
\item There exists a point $p_2 \in U_2$.
\item $U_2 \ni p_2$ and $V_2 \supset c_2(U_2) \ni c_2(p_2)$.
\item There exists a pair $({\phi}_{X,p},{\phi}_{Y,p})$ of $C^s$ diffeomorphisms satisfying ${\phi}_{Y,p} \circ c_1=c_2 \circ {\phi}_{X,p}$ such that the following properties hold.
\begin{itemize}
\item ${\phi}_{X,p}$ is a diffeomorphism from $U_1$ onto $U_2$.
\item ${\phi}_{Y,p}$ is a diffeomorphism from $V_1$ onto $V_2$.
\item ${\phi}_{X,p}(p_1)=p_2$.
\end{itemize}
\end{itemize}  

For {\it Morse} functions, see \cite{golubitskyguillemin} and \cite{milnor} for example. The former book respects the viewpoint from the singularity theory and the latter respects applications to algebraic topological or differential topological properties of the manifolds. 

\begin{Def}
A {\it fold} map is a $C^{\infty}$ map around each singular point which is $C^{\infty}$ equivalent to the product map of a Morse function and the identity map on a $C^{\infty}$ manifold. 
\end{Def}
Precise expositions on fold maps are in \cite{golubitskyguillemin} and \cite{saeki} for example. 

Hereafter, ${\mathbb{R}}^k$ denotes the $k$-dimensional Euclidean space endowed with the Euclidean metric. $||x|| \geq 0$ denotes the distance between $x$ and the origin $0$ in ${\mathbb{R}}^k$.

$S^k:=\{x \in {\mathbb{R}}^{k+1} \mid ||x||=1.\}$ is the $k$-dimensional unit sphere where $k$ is a positive integer. A copy or a smooth manifold $C^{\infty}$ diffeomorphic to $S^k$ is a $k$-dimensional {\it standard} sphere. 
$D^k:=\{x \in {\mathbb{R}}^{k} \mid ||x|| \leq 1.\}$ is the $k$-dimensional unit disk where $k$ is a positive integer. A copy or a smooth manifold $C^{\infty}$ diffeomorphic to $D^k$ is a $k$-dimensional {\it standard} disk. 

Morse functions are fold maps. A {\it height} function on (the interior of) a standard disk is one of simplest Morse functions. The function is a Morse function with exactly one singular point in the center of the disk. A {\it height} function on a standard (unit) sphere is also one of simplest Morse functions. This function on a standard sphere is a specific case of Morse functions on spheres with exactly two singular points, playing important roles in so-called Reeb's theorem.
\begin{Def}
\label{def:9}
A {\it special generic} map is a fold map around each singular point which is $C^{\infty}$ equivalent to the product map of a height function of the interior of a standard disk and the identity map on a $C^{\infty}$ manifold. 
\end{Def}
See \cite{saeki2} for special generic maps for example. For an integer $d \geq 1$, let $L_d$ be the $1$-dimensional polyhedron or the graph with $d+1$ vertices represented as $${\bigcup}_{k=0}^{d-1} \{(r \cos \frac{2k\pi}{d}, r \sin \frac{2k\pi}{d}) \mid 0 <r \leq 1 \} \bigcup \{(0,0)\} \subset {\mathbb{R}}^2.$$ where $(0,0)$ is also a vertex.
\begin{figure}
\includegraphics[width=25mm]{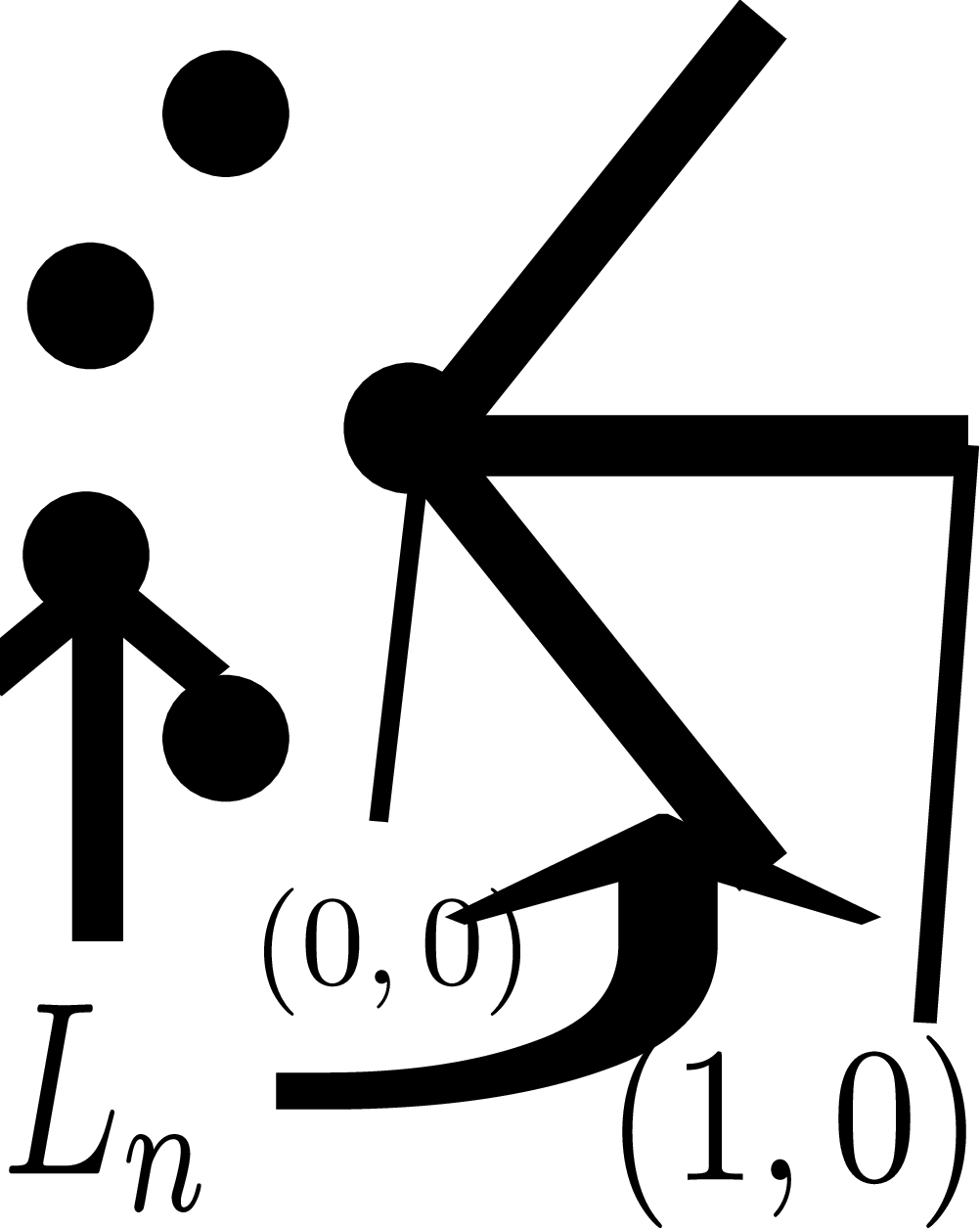}
\caption{$L_n \subset {\mathbb{R}}^2$.}
\label{fig:1}
\end{figure}

\begin{Def}
\label{def:10}
Let $d \geq 1$ be an integer. Let $q$ be a continuous map from a $C^s$ manifold of dimension $m>1$ onto $L_d \subset {\mathbb{R}^2}$. Let $r \geq 0$. $q$ is said to be {\it almost $C^s$ $D_d$-symmetric} if the following two are satisfied.
\begin{itemize}
\item For a transformation ${\phi}_{d}$ on $L_d$ defined as
$${\phi}_{d}(r \cos \frac{2k\pi}{d}, r \sin \frac{2k\pi}{d})=(r \cos \frac{2(k+1)\pi}{d}, r \sin \frac{2(k+1)\pi}{d})$$
there exists a $C^s$ diffeomorphism ${\Phi}_{d}$ satisfying ${\phi}_d \circ q=q \circ {\Phi}_d$.
\item For a transformation ${\phi}_{{\rm re}}$ on $L_d$ defined as $${\phi}_{{\rm re}}(r \cos \frac{2k\pi}{d}, r \sin \frac{2k\pi}{d})=(r \cos \frac{-2k\pi}{d}, r \sin \frac{-2k\pi}{d})$$
there exists a $C^s$ diffeomorphism ${\Phi}_{{\rm re}}$ satisfying ${\phi}_{{\rm re}} \circ q=q \circ {\Phi}_{{\rm re}}$.
\end{itemize}

\end{Def}
\begin{Thm}
\label{thm:1}
There exist a class $\mathcal{C}$ of $(C^{\infty},C^0)$ maps whose Reeb spaces are regarded as Reeb graphs and a class ${\mathcal{Q}}_{\mathcal{C}}$ of pseudo quotient maps of the class satisfying the following properties.
\begin{enumerate}
\item For maps of the class $\mathcal{C}$, preimages of regular values are disjoint unions of standard spheres. Moreover, the restriction of the map of the class to the preimage of a suitable small regular neighborhood of a vertex in the Reeb space is $C^{\infty}$ equivalent to the composition of a $(C^{\infty},C^0)$ map into the plane with a canonical projection to $\mathbb{R}$. 
\item For maps of the class ${\mathcal{Q}}_{\mathcal{C}}$, the restriction of the map to the preimage of a suitable small regular neighborhood of a vertex is an almost $C^{\infty}$ $D_d$-symmetric map onto $L_d$ by regarding the regular neighborhood as a graph consisting of exactly $d$ edges in a canonical way. 
\item For any finite and connected graph with at least one edge, we can construct a map of the class ${\mathcal{Q}}_{\mathcal{C}}$ onto the graph.
\item For a map of the class ${\mathcal{Q}}_{\mathcal{C}}$, if for the graph of the target, the degree of each vertex is at most $3$, then the map is realized 
as a quotient map of the class $\mathcal{C}$.
\end{enumerate} 
\end{Thm}
In the proof, we first construct a local function around each vertex in Steps 1, 2 and 3. 
In Step 4, we complete the construction by constructing remaining parts. Last, we give the definitions of $\mathcal{C}$ and $\mathcal{Q}_{\mathcal{C}}$. More rigorously, we only present conditions the classes should satisfy and we can see that this is sufficient to continue our discussions. We see that this completes the proof except the proof of the fourth property. Last we discuss the fourth property.
\begin{proof}
Step 1 Around a vertex of degree $2$. \\
We consider a trivial $C^{\infty}$ bundle over $[-1,1]$ whose fiber is a standard sphere. We compose a surjective function over $[-1,1]$ defined by $t \mapsto t^3$. $0 \in [-1,1]$ is identified as the vertex of degree $2$. We can regard that the remaining points are not in the vertex set. After composing a canonical embedding into ${\mathbb{R}}^2$, the map is regarded as an almost $C^{\infty}$ $D_2$-symmetric map onto $L_2$. \\
\ \\
Step 2 Around a vertex of degree $d \geq 3$. \\
We consider a $C^{\infty}$ map on an $m$-dimensional $C^{\infty}$ manifold into the plane whose image is the closure $D$ of the bounded domain surrounded by two segments and three curves including the curve represented by a parabola in the center in FIGURE \ref{fig:2}. We give expositions on the curve represented by the parabola. This is diffeomorphic to a line and unbounded in ${\mathbb{R}}^2$. 

We construct the $C^{\infty}$ map so that the following properties hold.
\begin{itemize}
\item The restriction to the singular set is an embedding.
\item The singular value set is the disjoint union of the two curves in the left and in the right of the half-space $\{(x_1,x_2)\mid x_2 \geq 0.\}$.
\item The curve represented by the parabola is diffeomorphic to a line and unbounded in ${\mathbb{R}}^2$. We can take a suitable open and connected subset $S_C$ satisfying the following conditions.
\begin{itemize}
\item $S_C$ is a smooth and connected curve bounded in ${\mathbb{R}}^2$.
\item $S_C$ contains the set of all points on the curve represented by the parabola being also in the boundary of $D$. The resulting set is a subset $S_C$. 
\item The restriction of the $C^{\infty}$ map on the $m$-dimensional manifold to the preimage of $S_C$ is $C^{\infty}$ equivalent to a Morse function with exactly two singular points on a $2$-dimensional standard sphere where the manifold of the target is taken as $S_C$.
\end{itemize}
\item Preimages of regular values of the $C^{\infty}$ map are standard spheres ($m>2$) or two-point sets ($m=2$). 
\end{itemize}
We explain about the composition of the map with a diffeomorphism again later.
For more precise facts on special generic maps into the plane, see \cite{saeki2} for example.
\begin{figure}
\includegraphics[width=30mm]{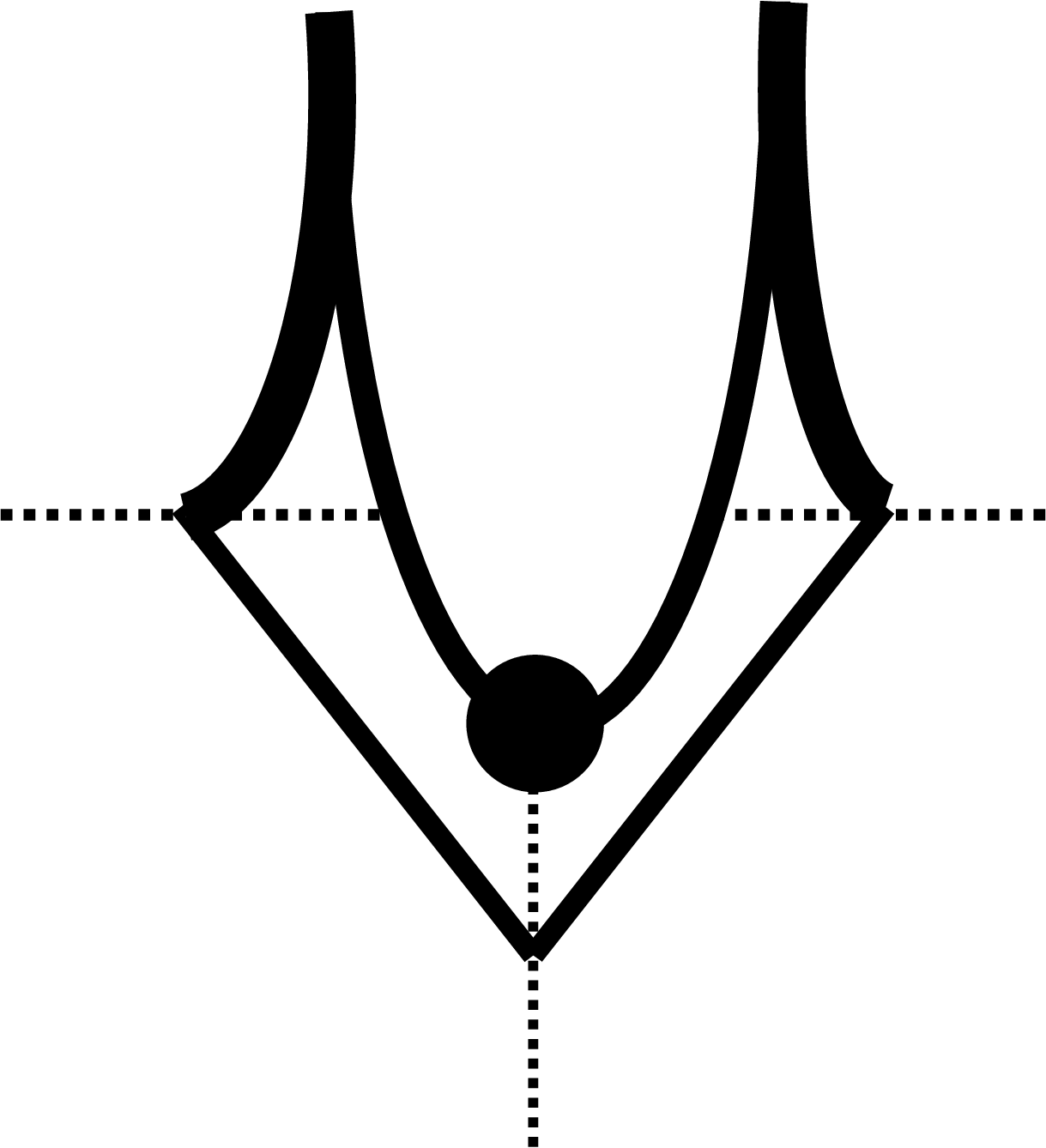}
\caption{The image of a special generic map into the plane.}
\label{fig:2}
\end{figure}
Let $t_0$ be a $C^{\infty}$ function whose value is $1$ on the interval $\{x \leq 0 \mid x \in \mathbb{R}.\}$ and which is strictly increasing on the interval $\{x \geq 0 \mid x \in \mathbb{R}.\}$. 

By setting the original $C^{\infty}$ map into the plane and $t_0$ suitably here, we can have the image of the composition of the original $C^{\infty}$ map with $T$ defined by
$$T(x_1,x_2):=
\begin{cases}
(x_1,x_2) & (x_2 \leq 0) \\
(t_0(x_2)x_1,x_2) & (x_2 \geq 0) 
\end{cases}
$$
as one represented as FIGURE \ref{fig:3}. Let us explain about the new bounded domain $T({\rm Int}\ D)$, the closure $T(D)$ and the resulting smooth map $f_{T,D}$. 
 
$T(D)$ is the closure of the bounded domain surrounded by four segments and one curve depicted in the figure and $p_0>0$ is given. The two segments containing $(0,-1)$ are $\{(\pm(-1+u),-u) \mid 0 \leq u \leq 1\}$. 
The curve, connecting the two thick segments $\{(\pm1,p) \mid 0 \leq p \leq p_0\}$, is defined as a subset of the quadratic curve $(x_2+1)^2-{x_1}^2={(c_0+1)}^2$, containing $(\pm1,p_0)$ for a suitable real number $c_0>-1$. 
The resulting smooth map into the plane satisfies the following properties.
\begin{itemize}
\item The restriction to the singular set is an embedding.
\item The singular value set is the disjoint union of the two thick segments $\{(\pm1,p) \mid 0 \leq p \leq p_0\}$.
\item We can take a suitable open and connected subset $T(S_C)$ which is a curve bounded in ${\mathbb{R}}^2$ and contains the subset, consisting of of all points on the quadratic curve being also in the boundary of $T(D)$ as a subset. We can also do this so that the restriction of the map to the preimage of $T(S_C)$ is $C^{\infty}$ equivalent to a Morse function with exactly two singular points on a $2$-dimensional standard sphere where the manifold of the target is taken as the curve $T(S_C)$. 
\item Preimages of regular values are standard spheres ($m>2$) or two-point sets ($m=2$). 
\item The restrictions to the preimages of the straight lines in ${\mathbb{R}}^2$ containing the two segments $\{(\pm(-1+u),-u) \mid 0 \leq u \leq 1\}$ are $C^{\infty}$ equivalent to a height function on the $2$-dimensional unit disk where the manifolds of the targets are taken as the straight lines.
\end{itemize}

We can determine $c(p_0)=c_0>-1$ by considering the point on the subset of the quadratic curve and we can naturally determine a $C^{\infty}$ function $c$ on $(0,p_0]$ respecting a natural family of quadratic curves ${(x_2+1)}^2-{x_1}^2={(c(p)+1)}^2$ for $-1 \leq c(p) \leq c_0$ by applying similar correspondences mapping $p \in [0,p_0]$ to $c(p)>-1$. $(0,c(p))$ is on the curve ${(x_2+1)}^2-{x_1}^2={(c(p)+1)}^2$ for $p \in (0,p_0]$. We can define a function $c$ mapping $0$ to $-1$ by extending the function on $(0,p_0]$. We can define a $(C^{\infty},C^0)$ map on the closure $T(D)$ of the bounded domain mapping points on the curves ${(x_2+1)}^2-{x_1}^2={(c(p)+1)}^2$ to $(0,c(p))$ and points on the two segments in the bottom to $(0,-1)$. Let $T_D$ denote the map.
The argument yields a $(C^{\infty},C^0)$ function on the given $m$-dimensional manifold to a closed interval $[-1,c(p_0)]$. This is defined by $T_D \circ f_{T,D}$ where we identify $(0,x_2)$ in the plane of the target with $x_2$ for $-1 \leq x_2 \leq c(p_0)$. 

\begin{figure}
\includegraphics[width=30mm]{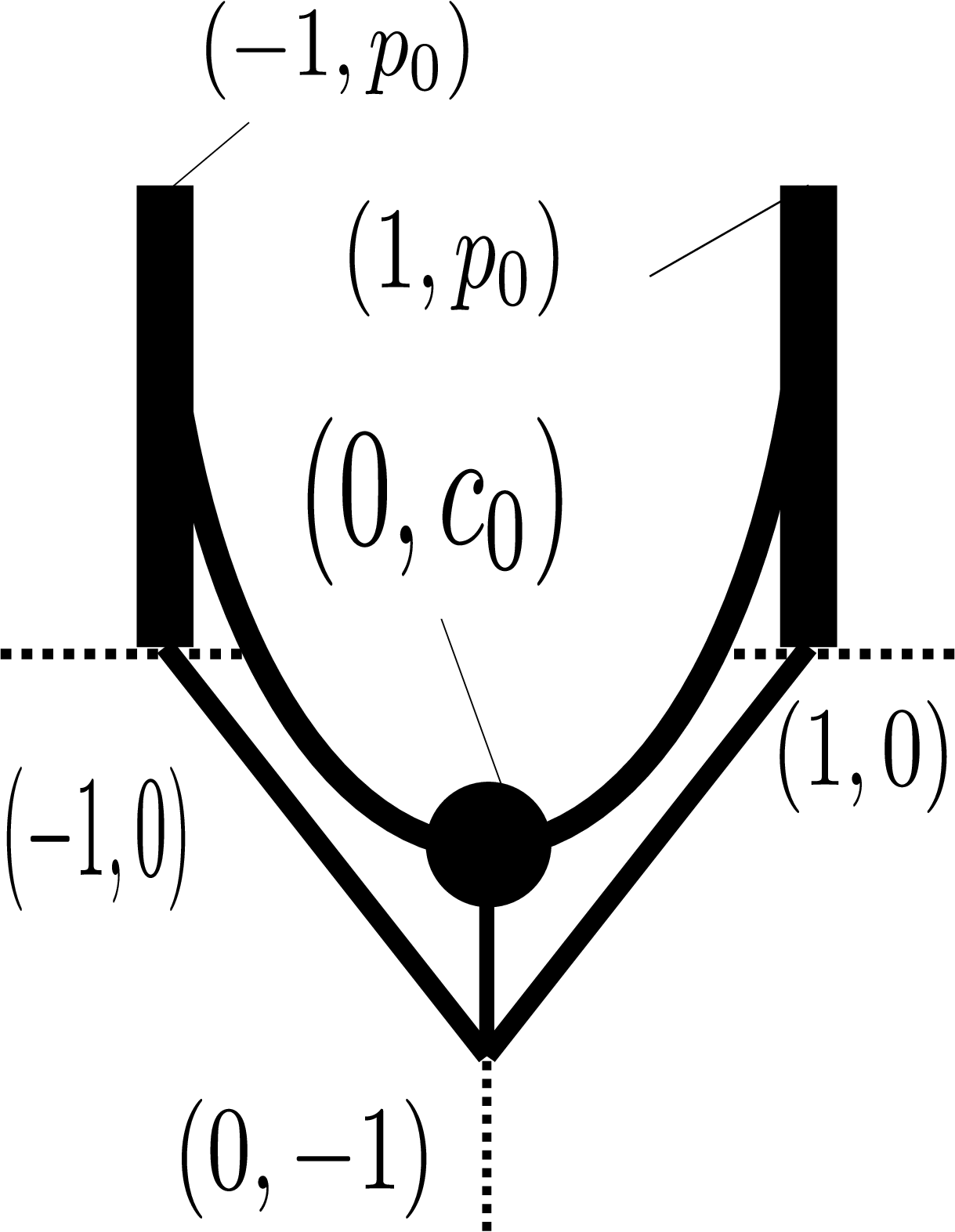}
\caption{The image of a deformation of the special generic map into the plane of FIGURE \ref{fig:2}.}
\label{fig:3}
\end{figure}

We consider the map $f_{T,D}$ on the $m$-dimensional manifold into the plane obtained in the explanation of FIGURE \ref{fig:2} and $d$ copies of this. We deform these maps by scaling suitably and attach these $d$ copies as shown in FIGURE \ref{fig:4} on the segments corresponding to ones including $(0,-1)$ in the original image and the preimages. $D_k$ stands for the images of the maps. Maps are also suitably scaled so that $(0,-1)$ goes to $(0,0)$ and that the angles formed by the pairs of the segments containing $(0,0)$ are equal and $\frac{2\pi}{d}$, for example. We can obtain an almost $C^{\infty}$ $D_d$-symmetric local map around the vertex such that preimages of points in the interiors of edges of the graph are standard spheres by composing maps playing roles $T_D$ has played before for each copy of the map $f_{T,D}$. 

The measure zero set of the resulting local map and the desired map we construct later is defined by taking the preimage of the union of the $d$ segments originating from the origin $0 \in {\mathbb{R}}^2$ and forming the $d$ angles before for each vertex of degree $d \geq 3$ and taking the disjoint union canonically. In Step 3, we construct another local map and this does not enlarge our measure zero set.
\begin{figure}
\includegraphics[width=30mm]{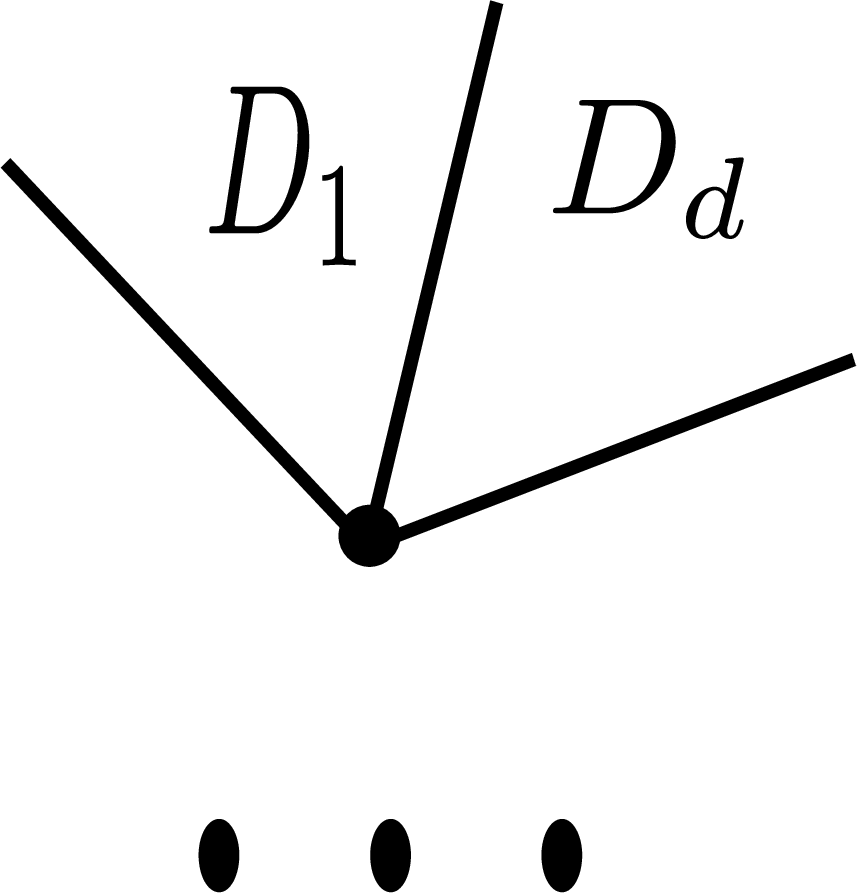}
\caption{Attacing $d$ copies of a map presented in FIGURE \ref{fig:2} (the image of the resulting map).}
\label{fig:4}
\end{figure}

\ \\
Step 3 Around a vertex of degree $1$. \\
We consider a natural height function $h$ on (the interior of) a copy of the $m$-dimensional unit disk whose image is $[0,1]$. The function is also a Morse function with exactly one singular point, which is the origin and in the center of the unit disk. The local map $q_h$ is easily seen as an almost $C^{\infty}$ $D_1$-symmetric map after composing a canonical embedding into ${\mathbb{R}}^2$. \\
\ \\
Step 4 Completing the construction. \\
For the interior of each edge, we construct a trivial $C^{\infty}$ bundle whose fiber is a standard sphere. Last we glue all the constructed local maps together to obtain a global map. More rigorously, we need to compose the resulting local map with a suitable PL homeomorphism onto a small regular neighborhood of each vertex in Steps 1--3 before gluing the local maps.  \\
\ \\
Step 5 Expositions on $\mathcal{C}$ and $\mathcal{Q}_{\mathcal{C}}$. \\
Through Steps 1--4, we construct a desired map for arbitrary finite and connected graph with at least one edge. Last, we explain about the classes $\mathcal{C}$ and $\mathcal{Q}_{\mathcal{C}}$. 

We first explain about $\mathcal{C}$. 
More precisely, we explain about conditions maps in this class should satisfy. This is sufficient to continue our discussions.
The restriction of the map in this class to the preimage of a suitable small regular neighborhood of a vertex of degree greater than $1$ is $C^{\infty}$ equivalent to a map obtained in the following way.
\begin{itemize}
\item Prepare a presented local map onto a regular neighborhood of a graph, regarded as a map into the plane where we compose the original map with the canonical embedding of the graph into the plane.
\item Compose the previous map with a homeomorphism on the plane satisfying the following properties (we define such a map as an {\it almost smooth generalized rotation with reflections}).
\begin{itemize}
\item The homeomorphism is $C^{\infty}$ on ${\mathbb{R}}^2-\{(0,0)\}$.
\item Each point expect $(0,0)$ is not a singular point of the restriction of the homeomorphism to ${\mathbb{R}}^2-\{(0,0)\}$.
\item For each point except $(0,0)$, the homeomorphism preserves the distance between the point and $(0,0)$.
\item The homeomorphism maps each straight line originating from $(0,0)$ to another straight line originating from $(0,0)$. 
\end{itemize}   
\item Compose the previous map with a canonical projection onto a straight line containing $(0,0)$ and intersecting no edges in the regular neighborhood of the graph of the target vertically.
\end{itemize}
Around the preimage of each vertex of degree $1$, a map in the class is a map such that the local form around the vertex is as in Step 3 or a natural height function on a unit disk. This completes the exposition on the class $\mathcal{C}$ and the proof except the proof of the fourth property.

For each finite graph which is not a single point or which has no vertices of degree greater than $3$, we can give an orientation to each edge of the graph so that we can construct a continuous map from the graph into $S^1$ satisfying the following two.
\begin{itemize}
\item On each edge the map is injective.
\item The orientation of each edge canonically induced from a canonical orientation of $S^1$ coincides with the given orientation. 
\end{itemize}
If the graph has no loop, then we can replace $S^1$ by $\mathbb{R}$.

For a map of the class $\mathcal{Q}_{\mathcal{C}}$, if for the graph of the target and each vertex, the degree is at most $3$, then we can orient the graph as this and can construct a local function respecting the definition of the class $\mathcal{C}$ and the orientations of edges. This is due to the definitions of a $D_2$-symmetric and a $D_3$-symmetric map and an almost smooth generalized rotation with reflections. We can consider a transformation by an almost smooth generalized rotation with reflections to construct a local function compatible with the desired orientations of the edges. See FIGURE \ref{fig:4.5} for the case of a vertex of degree $2$ for example. We can glue local functions to obtain a desired function.

\begin{figure}
\includegraphics[width=30mm]{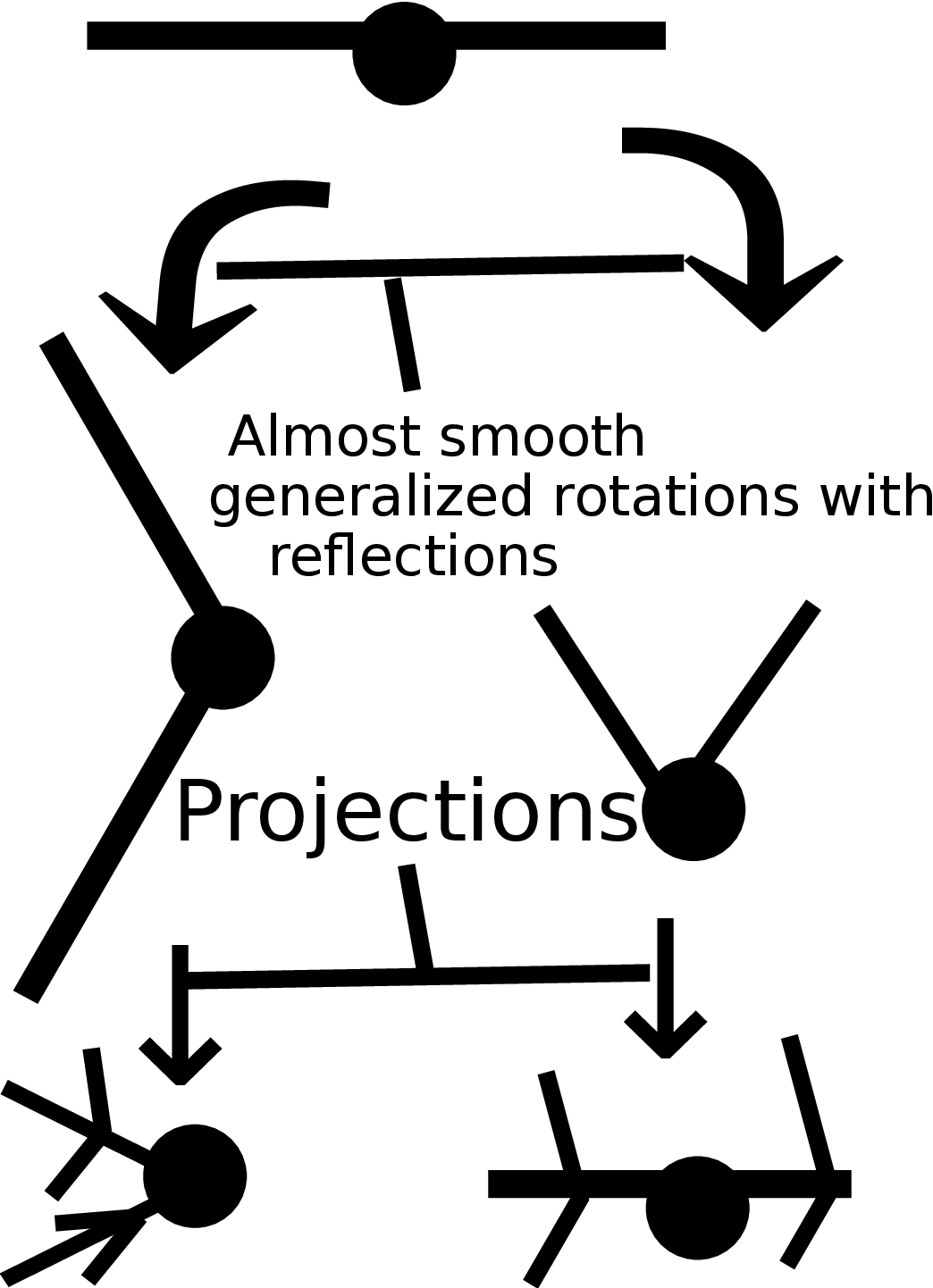}
\caption{Graphs of the targets of maps and almost smooth generalized rotations with reflections and canonical projections (around a vertex of degree $2$: note that the graphs in the bottom are the figure representing Reeb graphs of the local functions locally and that arrows indicate canonical local orientations of the graphs induced naturally from the local functions respecting the values).}
\label{fig:4.5}
\end{figure}

This completes the proof.

\end{proof}
We present another example of classes of maps and pseudo quotient maps of the class. 
\begin{Ex}
A {\it standard-spherical} Morse function is a Morse function such that the following properties hold (\cite{kitazawa}).
\begin{itemize}
\item At distinct singular points the (singular) values are distinct.
\item Preimages of regular values are disjoint unions of finite copies of standard spheres.
\item A vertex of the Reeb graph such that the preimage contains a singular point at which the function does not have a local extremum is a vertex of degree $3$.
\end{itemize}
We consider pseudo quotient maps of the class of such functions. We regard these functions as $(C^{\infty},C^s)$ functions whose measure zero sets are empty.

We present local forms of these pseudo quotient maps with several preimages in FIGURE \ref{fig:5}. 
\begin{figure}
\includegraphics[width=30mm]{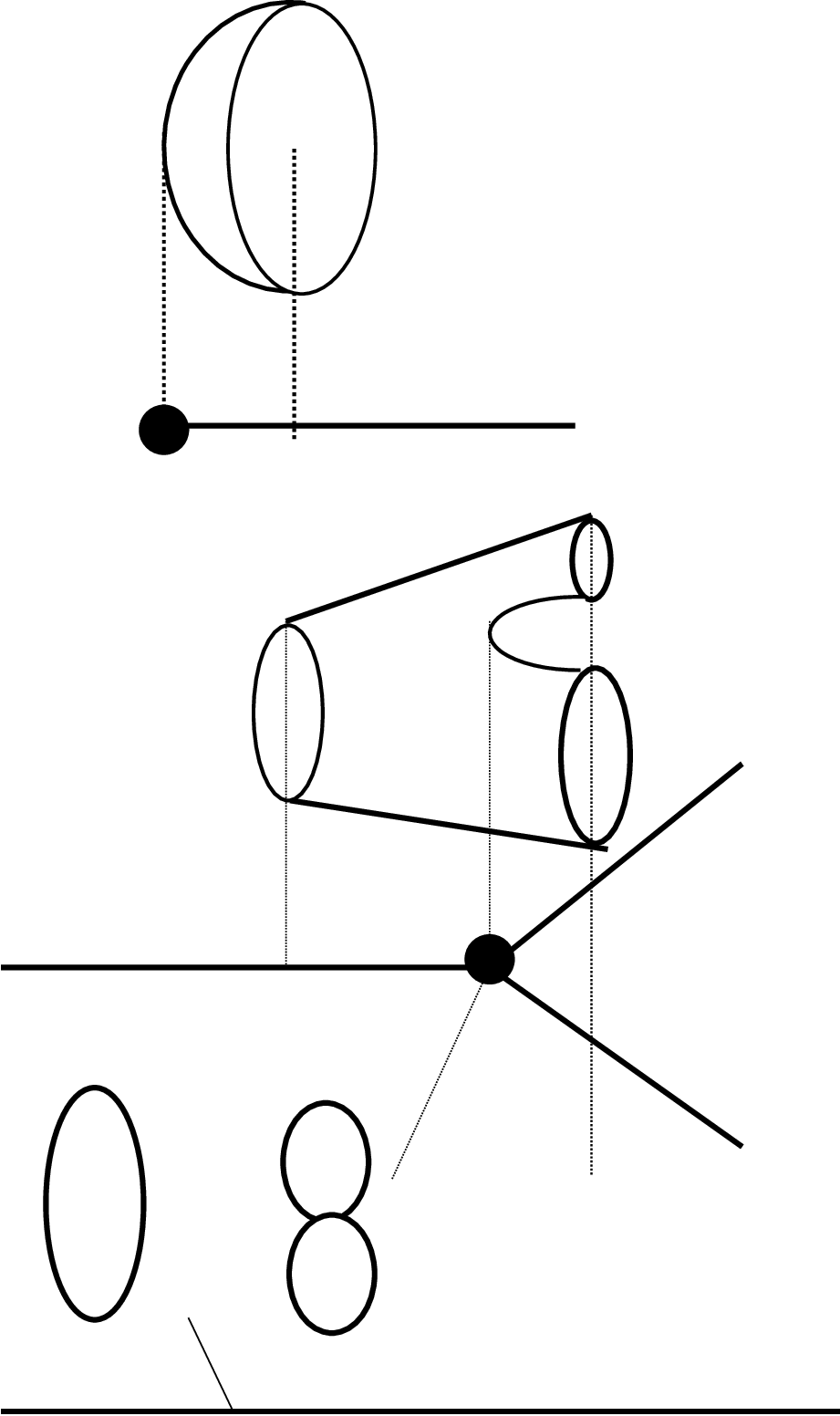}
\caption{Local forms of the pseudo quotient maps with several preimages.}
\label{fig:5}
\end{figure}

We investigate the local form around a vertex of degree $3$ of a pseudo quotient map of the class of the functions and the preimage. Consider an arbitrary small regular neighborhood (of the one-point set) of the only one singular point, which is also in the preimage of the vertex. 
See also FIGURE \ref{fig:6}.

\begin{figure}
\includegraphics[width=15mm]{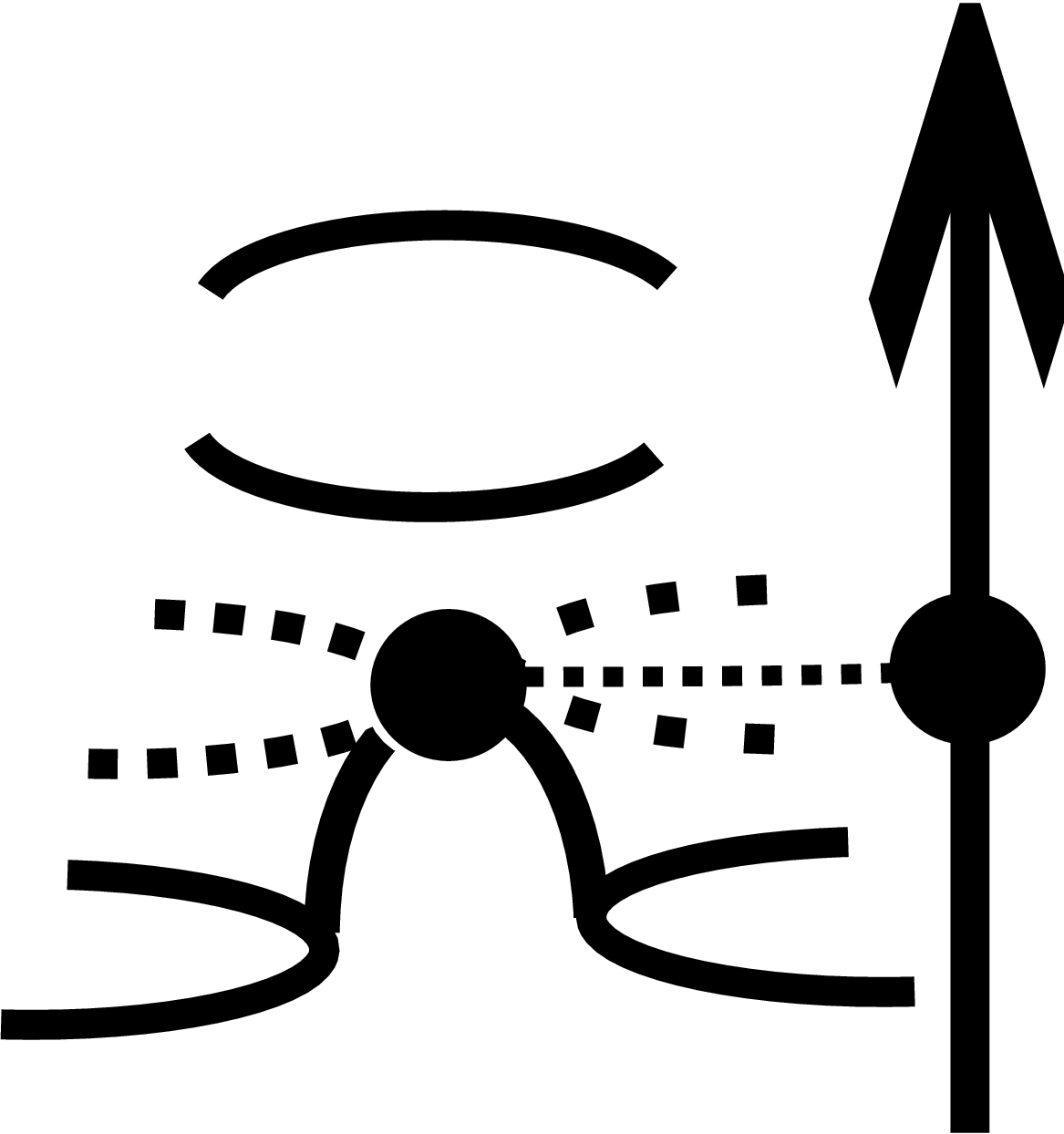}
\caption{The local form around a vertex of degree $3$ of a pseudo quotient map of the class of standard-spherical functions. An arbitrary regular neighborhood of (the one-point set consisting of) the only one singular point is depicted.}
\label{fig:6}
\end{figure}

If we remove the intersection of the preimage of the vertex and the arbitrary small regular neighborhood of (the one-point set consisting of) the only one singular point from the small regular neighborhood, then the resulting space has exactly $4$ connected components. Two of the connected components are in the upper part and the others are in the lower part. Moreover, the former two connected components are mapped onto an interval of the form $(a_{\rm vertex},a_{\rm up}]$ and the latter two connected components are mapped onto the disjoint union of two intervals of the form $[a_{\rm low},a_{\rm vertex})$ in the graph by the quotient map to the Reeb space.

This yields the fact that a pseudo quotient map the graph of whose target is as FIGURE \ref{fig:7} cannot be realized as a quotient map of the class. Arrows indicate natural orientations induced from canonically obtained local functions respecting their values. Note also that around a vertex of degree $3$, we cannot regard the local map as an almost $C^s$ $D_3$-symmetric map for any integer $s$ or $s=\infty$.

\begin{figure}
\begin{center}
\includegraphics[width=10mm]{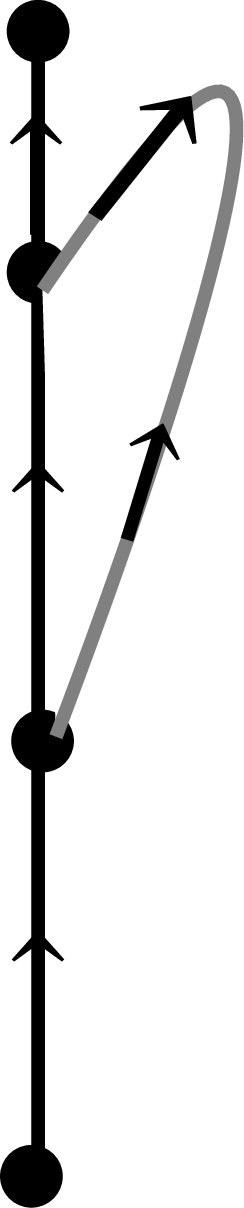}
\end{center}
\caption{A graph representing the graph of the target of a pseudo quotient map of the class of standard-spherical functions. Arrows represent orientations on edges of the graph canonically induced from local functions respecting their values. We can not give a suitable orientation to the curved edge.}
\label{fig:7}
\end{figure}

Local forms of these standard-spherical functions are also discussed in \cite{saeki3} for example.

Last, compare this case with Theorem \ref{thm:1}.
\end{Ex}

\begin{Thm}
\label{thm:2}
Let $m>2$ be a positive integer.
Let $\mathcal{C}$ be a class of $(C^r,C^0)$ maps from $m$-dimensional differentiable manifolds onto $1$-dimensional ones whose Reeb spaces are regarded as Reeb graphs. Moreover, the restriction of the map to the preimage of a suitable small regular neighborhood of each vertex is $C^r$ equivalent to a map obtained by composing the following three maps in order.
\begin{itemize}
\item The $(C^{r},C^{0})$ map itself into the plane whose image is $L_n$ as FIGURE \ref{fig:1} for some $n \geq 1$.
\item An almost smooth generalized rotation with reflections.
\item A canonical projection onto a straight line containing the origin $(0,0)$ and containing no other point in edges of the graph.
\end{itemize}
In this situation, there exists a class $\mathcal{C^{\prime}}$ of $(C^r,C^0)$ maps from $m$-dimensional differentiable manifolds into $1$-dimensional ones equal to or greater than $\mathcal{C}$ such that the following properties hold.
\begin{enumerate}
\item
\label{thm:2.1}
The Reeb spaces are regarded as Reeb graphs.
\item
\label{thm:2.2}
A pseudo quotient map of the original class is also of this class and the converse holds.
\item 
\label{thm:2.3}
A pseudo quotient map of this class is always realized 
as a quotient map of the class $\mathcal{C^{\prime}}$.
\item
\label{thm:2.4}
For any map of this class, on the preimage of a suitable small regular neighborhood of each vertex in the Reeb graph, it is $C^r$ equivalent to a map obtained by composing the following three maps in order.
\begin{itemize}
\item The $(C^{r},C^{0})$ map itself into the plane whose image is $L_n$, explained in the assumption
\item A $(C^{\infty},C^0)$ map from the plane into the ${\mathbb{R}}^3$.
\item A canonical projection onto $\mathbb{R}$, defined by ${\pi}_{3,1}(x_1,x_2,x_3):=x_3$. 
\end{itemize}
\end{enumerate}
\end{Thm}
\begin{proof}
On the preimage of a small and connected open neighborhood of each vertex, the map of the class $\mathcal{C}$ is represented as or $C^r$ equivalent to a map obtained by composing the following three maps in order.
\begin{itemize}
\item A $(C^r,C^0)$ map into the plane explained in the assumption whose image is $L_n$ as FIGURE \ref{fig:1}.
\item An almost smooth generalized rotation with reflections.
\item The canonical projection onto a straight line containing the origin $(0,0)$ and containing no other points in the graph of the target. 
\end{itemize}
We revise this for the desired class $\mathcal{C^{\prime}}$. First we replace "an almost smooth generalized rotation with reflections" by a "$(C^{\infty},C^0)$ map into ${\mathbb{R}}^2 \times \mathbb{R}$ whose measure zero set is $\{(0,0)\}$". We explain this map in the last. Second, we replace "canonical projection onto a straight line containing the origin $(0,0)$ and containing no other points in the graph of the target" by "canonical projection from ${\mathbb{R}}^2 \times \mathbb{R}$ onto $\mathbb{R}$, defined by ${\pi}_{3,1}(x_1,x_2,x_3):=x_3$".

We present a $(C^{\infty},C^0)$ map from the plane into ${\mathbb{R}}^2 \times \mathbb{R}$ first. We consider an arbitrary map $l_{d,1,-1}$ from the set of all integers from $1$ to $d>0$ into $\{-1,1\}$. We can define a desired map $e$ satisfying the following properties.
\begin{itemize}
\item $e((0,0))=((0,0),0) \in {\mathbb{R}}^2 \times \mathbb{R}={\mathbb{R}}^3$.
\item $e((r \cos \frac{2(k+1)\pi}{d}, r \sin \frac{2(k+1)\pi}{d}))=((r \cos \frac{2(k+1)\pi}{d}, r \sin \frac{2(k+1)\pi}{d}),l_{d,1,-1}(k+1)r)$ for each integer $0 \leq k<d$ and $r>0$.
\item $\{(0,0)\}$ can be taken as the measure zero set of $e$. 
\end{itemize}
This is regarded as a piecewise smooth (PL) embedding. 

The first desired property (\ref{thm:2.1}) is obvious and this with the properties of the canonical projection yields the second desired property (\ref{thm:2.2}).

This yields a desired local function yielding natural orientations to edges of the Reeb graph. Moreover, we can give arbitrary orientations to the edges of the graph by choosing a suitable map $l_{d,1,-1}$. Remember the proof of the last property of Theorem \ref{thm:1}. We can construct a continuous map from the graph into $S^1$ satisfying the following two if we give orientations to the edges in some suitable way.
\begin{itemize}
\item On each edge the map is injective.
\item The orientation of each edge canonically induced from a canonical orientation of $S^1$ coincides with the given orientation. 
\end{itemize}

If the graph has no loop, then we can replace $S^1$ by $\mathbb{R}$.
By discussions similar to some discussions on a class of $(C^\infty,C^0)$ maps in the proof of Theorem \ref{thm:1}, we can see that the new class $\mathcal{C^{\prime}}$ can be regarded as a desired class satisfying the four desired properties. This completes the proof.
\end{proof}
To $\mathcal{C}$ in Theorem \ref{thm:1}, we can apply this for example. 

%
  

\begin{Rem}
\label{rem:1}
We consider {\it pseudo quotient} maps to not only graphs but also general polyhedra similarly. We can define notions similarly in these cases. 
In \cite{kobayashisaeki} such notions are already introduced for smooth maps of some classes on closed $C^{\infty}$ manifolds whose dimensions are greater than $2$ into the plane.

It is well-known that so-called {\it {\rm (}$C^{\infty}${\rm )} stable} maps form a large class of so-called {\it generic} $C^{\infty}$ maps between $C^{\infty}$ manifolds. We omit the definitions and for the definitions and fundamental theory see \cite{golubitskyguillemin} for example. A $C^{\infty}$ function on a closed $C^{\infty}$ manifold is known to be stable if and only if it is a Morse function such that at distinct singular points, the (singular) values are distinct. Most of functions in the present paper are not ($C^{\infty}$) stable.
  
In \cite{kobayashisaeki}, for a stable map on closed $C^{\infty}$ manifolds whose dimensions are greater than $2$ into the plane, the Reeb space is shown to be a $2$-dimensional polyhedron. \cite{shiota} generalizes the result in a more sophisticated way. Moreover, a problem similar to ones solved in the present paper is considered in \cite{kobayashisaeki}. More precisely, a pseudo quotient map (satisfying several conditions) of a suitable class of {\it stable} maps from closed $C^{\infty}$ manifolds whose dimensions are greater than $2$ into the plane is shown to be realized as a quotient map of this class of stable maps in Proposition 5.2 there.
\end{Rem}

\begin{Rem}
\label{rem:2}
We do not know whether we can construct desired functions in Theorem \ref{thm:1} and Theorem \ref{thm:2} as $C^{\infty}$ functions. This requires us to introduce $(C^r,C^s)$ maps as new notions. 
\end{Rem}

\section{Acknowledgement.}
The author would like to thank Professor Irina Gelbukh for comments on the preprints \cite{kitazawa4} and \cite{kitazawa5} by the author. Her comments and discussions with Osamu Saeki closely related to \cite{saeki4} and the two articles have mainly motivated the author to continue new studies on these papers including the present one.

The author is a member of JSPS KAKENHI Grant Number JP17H06128 "Innovative research of geometric topology and singularities of differentiable mappings"
(https://kaken.nii.ac.jp/en/grant/KAKENHI-PROJECT-17H06128/: Principal Investigator is Osamu Saeki). This work is conducted supported by this project. We declare that all data supporting the present study are in the present paper.


\end{document}